\documentclass[11pt, a4paper, reqno]{amsart}
\usepackage[utf8]{inputenc}
\usepackage[shortalphabetic]{amsrefs}
\usepackage{enumerate}

\theoremstyle{plain}
\newtheorem{theorem}{Theorem}[section]

\newtheorem{lemma}[theorem]{Lemma}

\theoremstyle{definition}
\newtheorem{definition}[theorem]{Definition}
\theoremstyle{remark}
\newtheorem{remark}[theorem]{Remark}
\newtheorem{example}[theorem]{Example}

\newcommand{\Z}{Z}

\newcommand{\C}{\mathbb{C}}
\newcommand{\Zp}{{\mathbb{Z}_{\geq 0}}}

\newcommand{\Uqsl}{U_q \mathfrak{sl}_2}
\newcommand{\Uqsu}{U_q \mathfrak{su}_2}

\newcommand{\CSLq}{\mathbb{C}[SL_2]_q}
\newcommand{\CSUq}{\mathbb{C}[SU_2]_q}

\newcommand{\sym}{\textnormal{sym}}
\newcommand{\Dsym}{\mathbb{D}_2^\sym}
\newcommand{\detqsym}{\textnormal{det}_q^\sym}
\newcommand{\Matsym}{\textnormal{Mat}_2^\sym}

\newcommand{\CMatsymq}{\mathbb{C}[\Matsym]_q}
\newcommand{\PolMatsymq}{\textnormal{Pol}(\Matsym)_q}
\newcommand{\PolCq}{\textnormal{Pol}(\C)_{q^2}}

\newcommand{\contq}{C(\Dsym)_q}
\newcommand{\holoq}{A(\mathbb{D}_2^\sym)_q}
\newcommand{\shilovcontq}{C(S(\Dsym))_q}
\newcommand{\shilovregq}{\C[S(\Dsym)]_q}

\title[Shilov boundary for a $q$-analog of holomorphic functions]{The Shilov boundary for a $q$-analog of the holomorphic functions on the unit ball of $2 \times 2$ symmetric matrices}

\author{Jimmy Johansson}
\author{Lyudmila Turowska}
\address{Department of Mathematical Sciences,
Chalmers University of Technology and the University of Gothenburg,
Gothenburg SE-412 96, Sweden}
\email{gusjohaji@student.gu.se}
\email{turowska@chalmers.se}
\date{\today}

\begin{document}
\begin{abstract}
We describe the Shilov boundary for a $q$-analog of the algebra of holomorphic functions on the unit ball in the space of symmetric $2 \times 2$ matrices.
\end{abstract}
\maketitle
\section{Introduction}
In the middle of the 1990s, L. Vaksman initiated a program to develop a $q$-analog of the theory of holomorphic functions on bounded symmetric domains (see \cite{vaksman-book} and references therein).
Among the numerous results which have emanated under this program we shall in this paper be interested in a noncommutative analog of the maximum modulus principle, a notion whose foundation is comprised of a noncommutative generalization of the Shilov boundary in the setting of operator algebras, which was developed by W. Arveson in \cites{arveson1, arveson2}.

In \cite{vaksman-boundary}, Vaksman proved a $q$-analog of the maximum modulus principle for the unit polydisk in $\C^n$, and more recently D. Proskurin and L. Turowska obtained, in \cite{pro-tur}, an analogous result for the unit ball in the space of $2 \times 2$ matrices. In this paper we show that similar methods can be used to compute the Shilov boundary ideal for a $q$-analog of the algebra of holomorphic functions on the unit ball in the space of symmetric $2 \times 2$ matrices.

The paper is organized as follows. In Section 2 we collect some basic material from the theory of quantum groups that we will need in this paper. In Section 3 we introduce the algebra of polynomials on quantum complex symmetric $2\times 2$ matrices and  discuss its universal enveloping $C^*$-algebra $C(\Dsym)_q$, a $q$-analog of the continuous functions on the unit ball $\Dsym = \{ \Z \in \Matsym: \Z^* \Z \leq I \}$. We  prove, in particular, that the Fock representation is a faithful irreducible representation of  $C(\Dsym)_q$. In Section 4 we describe the Shilov boundary ideal for the closed subalgebra $A(\Dsym)_q$, a $q$-analog of the algebra of functions holomorphic on the open unit ball of $\Matsym$ and continuous on its closure. The key tool, like in \cite{vaksman-boundary} and \cite{pro-tur}, is a unitary dilation of a contractive operator on a Hilbert space. Finally, in Section 5, we show that our result agrees with the definition of  a $*$-algebra referred to as the algebra of regular functions on the Shilov boundary, whose definition was proposed in \cite{bershtein-2}.

In this paper all algebras are assumed to be associative unital algebras over $\mathbb{C}$ and  $q \in (0, 1)$.

\section{Preliminaries}\label{section2}
In this section we review and fix our notation for the notions from the theory of quantum groups that we shall employ in this paper.

The algebra $\CSLq$ is defined by the generators $t_{i j}$, $i, j = 1, 2$, and the relations
\begin{align*}
	&t_{1 1} t_{2 1} = q t_{2 1} t_{1 1}, \quad
	t_{1 1} t_{1 2} = q t_{1 2} t_{1 1}, \quad
	t_{1 2} t_{2 1} = t_{2 1} t_{1 2} \\
	&t_{2 2} t_{2 1} = q^{-1} t_{2 1} t_{1 1}, \quad
	t_{2 2} t_{1 2} = q^{-1} t_{1 2} t_{2 2} \\
	&t_{1 1} t_{2 2} - t_{2 2} t_{1 1} = (q - q^{-1}) t_{1 2} t_{2 1}, \quad
	t_{1 1} t_{2 2} - q t_{1 2} t_{2 1} = 1.
\end{align*}
We define $\CSUq = (\CSLq, *)$, where the involution $*$ is determined by $t_{1 1}^* = t_{2 2}$ and $t_{1 2}^* = -q t_{2 1}$. 

Here and throughout this paper we denote by $\{ e_k: k \in \Zp \}$ the standard orthonormal basis for the Hilbert space $\ell^2(\Zp)$, and we let $S$, $C_n$, $D \in \mathcal{B}(\ell^2(\Zp))$ denote the operators defined by
\begin{equation}
\label{eq:operators}
	S e_k = e_{k + 1}, \quad
	C_n e_k = \sqrt{1 - q^{n k}} e_k, \quad
	D e_k = q^k e_k.
\end{equation}
It is well known that $\mathbb{C}[SU_2]_q$ admits the irreducible representations $\pi_\varphi$, $\varphi \in [0, 2 \pi)$, acting on $\ell^2(\Zp)$, which are determined by
\begin{align}
\begin{split}
\label{suq2}
	\pi_\varphi(t_{1 1}) &= S^* C_2, \quad \pi_\varphi(t_{1 2}) = -q e^{-i \varphi} D \\
	\pi_\varphi(t_{2 1}) &= e^{i \varphi} D, \quad \pi_\varphi(t_{2 2}) = C_2 S.
\end{split}
\end{align}

$\CSUq$ can also be equipped with a Hopf $*$-algebra structure (see e.g.~\cite{klimyk_schmudgen}).
In particular, the comultiplication is given by
\[
	\Delta(t_{i j}) = \sum_{k = 1}^2 t_{i k} \otimes t_{k j}, \quad i,j=1,2.
\]

We denote by $\Uqsl$ the Hopf algebra generated by $E, F, K, K^{-1}$ satisfying the relations
\[
	K K^{-1} = K^{-1} K = 1, \quad
	K E = q^2 E K, \quad
	K F = q^{-2} F K
\]
\[
	[E, F] = \frac{K - K^{-1}}{q - q^{-1}}.
\]
The comultiplication $\Delta$, the antipode $S$, and the the counit $\varepsilon$ are defined by
\[
	\Delta(E) = E \otimes 1 + K \otimes E, \quad
	\Delta(F) = F \otimes K^{-1} + 1 \otimes F, \quad
	\Delta(K) = K \otimes K
\]
\[
	S(E) = -K^{-1} E, \quad
	S(F) = -F K, \quad
	S(K) = K^{-1}
\]
\[
	\varepsilon(E) = \varepsilon(F) = 0, \quad
	\varepsilon(K) = 1.
\]
We let $\Uqsu$ denote the Hopf $*$-algebra $(\Uqsl, *)$, where the involution is given by
\[
	E^* = K F, \quad
	F^* = E K^{-1}, \quad
	K^* = K.
\]
We recall that $\CSLq$ is the finite dual of $\Uqsl$. As linear functionals the elements of $\CSLq$ are determined by
\begin{equation}
\label{eq:CSLq-functionals}
	t_{1 2}(E) = q^{-1/2}, \quad
	t_{2 1}(F) = q^{1/2}, \quad
	t_{1 1}(K) = q, \quad
	t_{2 2}(K) = q^{-1}
\end{equation}
and all other evaluations on the generators are zero.

We shall also need the $*$-algebra $\PolCq$, a $q$-analog of the $*$-algebra of polynomials on $\C$, which is defined by the generator $z$ and the relation $z^* z = q^4 z z^* + 1 - q^4$.

We have the following list of irreducible representations of $\PolCq$, up to unitary equivalence (see \cite{pusz-woronowicz}):
\begin{enumerate}[(i)]
\item
the Fock representation $\rho_F$ acting on $\ell^2(\Zp)$: $\rho_F(z) = C_4 S$;
\item
one-dimensional representations $\rho_\varphi$, $\varphi \in [0, 2 \pi)$: $\rho_\varphi(z) = e^{i \varphi}$.
\end{enumerate}
\section{A $q$-analog of the algebra of continuous and holomorphic functions on the unit ball}
The algebra $\CMatsymq$ is defined by the generators $z_{1 1}$, $z_{2 1}$, $z_{2 2}$ satisfying the relations
\begin{align}
	&z_{1 1} z_{2 1} = q^2 z_{2 1} z_{1 1} \nonumber\\
\label{eq:CMatsymq-2}
	&z_{2 1} z_{2 2} = q^2 z_{2 2} z_{2 1} \\
	&z_{1 1} z_{2 2} - z_{2 2} z_{1 1} = q (q^2 - q^{-2}) z_{2 1}^2.\nonumber
\end{align}
The algebra admits a natural gradation given by $\deg z_{ij}=1$.
The $*$-algebra $\PolMatsymq$, a $q$-analog of the $*$-algebra of polynomials on the space of symmetric complex $2 \times 2$ matrices, is defined by the generators $z_{1 1}$, $z_{2 1}$, $z_{2 2}$ satisfying the relations~\eqref{eq:CMatsymq-2}
and
\begin{align}
	z_{1 1}^* z_{1 1} &= q^4 z_{1 1} z_{1 1}^* -
	q (q^{-1} - q) (1 + q^2)^2 z_{2 1} z_{2 1}^* + \nonumber\\
	&(q^{-1} - q)^2 (1 + q^2) z_{2 2} z_{2 2}^* + 1 - q^4 \nonumber\\
	z_{1 1}^* z_{2 1} &= q^2 z_{2 1} z_{1 1}^* -
	q (q^{-1} - q) (q^{-1} + q) z_{2 2} z_{2 1}^* \nonumber\\
	\label{eq:PolMatsymq-last}
  z_{1 1}^* z_{2 2} &= z_{2 2} z_{1 1}^* \\
	z_{2 1}^* z_{2 1} &= q^2 z_{2 1} z_{2 1}^* - (1 - q^2) z_{2 2} z_{2 2}^* + 1 - q^2 \nonumber\\
	z_{2 1}^* z_{2 2} &= q^2 z_{2 2} z_{2 1}^* \nonumber\\
z_{2 2}^* z_{2 2} &= q^4 z_{2 2} z_{2 2}^* + 1 - q^4.\nonumber
\end{align}
\begin{remark}
For the sake of symmetry and for brevity in formulas (see e.g. Lemma~\ref{lemma:homomorphism}), one may include $z_{1 2}$ as an additional generator together with the relation $z_{1 2} = q z_{2 1}$.
\end{remark}
We have that $\CMatsymq$ is a $\Uqsl$-module algebra, where the $\Uqsl$-action is given as follows (\cite{bershtein-2}):
\begin{equation}
\label{eq:E-action}
	E z_{i j} = q^{-1/2}
	\begin{cases}
	0, &i = j = 1 \\
	z_{1 1}, &i = 2, \, j = 1 \\
	(q + q^{-1}) z_{2 1}, &i = j = 2
	\end{cases}
\end{equation}
\begin{equation}
\label{eq:F-action}
	F z_{i j} = q^{1/2}
	\begin{cases}
	(q + q^{-1}) z_{2 1}, &i = j = 1 \\
	z_{2 2}, &i = 2, \, j = 1 \\
	0, &i = j = 2
	\end{cases}
\end{equation}
\begin{equation}
\label{eq:K-action}
	K z_{i j} =
	\begin{cases}
	q^2 z_{1 1}, &i = j = 1 \\
	z_{2 1}, &i = 2, \, j = 1 \\
	q^{-2} z_{2 2}, &i = j = 2.
	\end{cases}
\end{equation}
Recall that the action of $\Uqsl$ on other elements of $\CMatsymq$ can be obtained from the property that
\[
	\xi(f g) = \sum_i(\xi_i^{(1)}f)(\xi_i^{(2)}g)
\]
for $\xi \in \Uqsl$, $f,g\in \CMatsymq$ and $\Delta(\xi)=\sum_i\xi_i^{(1)}\otimes\xi_i^{(2)}$ (in the Sweedler notation).

Since the involutions in $\Uqsu$ and $\PolMatsymq$ are compatible in the sense that
\[
	(\xi f)^* = S(\xi)^* f^*, \quad \xi \in \Uqsu, \, f \in \PolMatsymq,
\]
the action of $\Uqsl$ on $\CMatsymq$ can be extended to an action of $\Uqsu$ on $\PolMatsymq$. Explicitly, the $\Uqsu$-action is given by~\eqref{eq:E-action}--\eqref{eq:K-action} together with
\[
	E z_{i j}^* = -q^{-2} (F z_{i j})^*, \quad
	F z_{i j}^* = -q^2 (E z_{i j})^*, \quad
	K z_{i j}^* = (K^{-1} z_{i j})^*.
\]

The irreducible representations of $\PolMatsymq$, which we present in the following theorem, were classified in \cite{bershtein-1}.
\begin{theorem}
The irreducible representations of $\PolMatsymq$ up to unitary equivalence are given by
\begin{enumerate}[(i)]
\item
the Fock representation acting on $\ell^2(\Zp)^{\otimes 3}$:
\begin{align*}
	\pi_F(z_{1 1}) &= I \otimes D^2 \otimes C_4 S -
	q^{-1} S^* C_4 \otimes C_2 S C_2 S \otimes I \\
	\pi_F(z_{2 1}) &= D^2 \otimes C_2 S \otimes I \\
	\pi_F(z_{2 2}) &= C_4 S \otimes I \otimes I;
\end{align*}
\item
representations $\tau_\varphi$, $\varphi \in [0, 2 \pi)$, acting on $\ell^2(\Zp)^{\otimes 2}$:
\begin{align*}
	\tau_\varphi(z_{1 1}) &= e^{i \varphi} I \otimes D^2 -
	q^{-1} S^* C_4 \otimes C_2 S C_2 S \\
	\tau_\varphi(z_{2 1}) &= D^2 \otimes C_2 S \\
	\tau_\varphi(z_{2 2}) &= C_4 S \otimes I;
\end{align*}
\item
representations $\omega_\varphi$, $\varphi \in [0, 2 \pi)$, acting on $\ell^2(\Zp)$:
\begin{align*}
	\omega_\varphi(z_{1 1}) &= -q^{-1} e^{2 i \varphi} S^* C_4 \\
	\omega_\varphi(z_{2 1}) &= e^{i \varphi} D^2 \\
	\omega_\varphi(z_{2 2}) &= C_4 S;
\end{align*}
\item
representations $\nu_\varphi$, $\varphi \in [0, 2 \pi)$, acting on $\ell^2(\Zp)$:
\begin{align*}
	\nu_\varphi(z_{1 1}) &= q^{-1} C_4 S \\
	\nu_\varphi(z_{2 1}) &= 0 \\
	\nu_\varphi(z_{2 2}) &= e^{i \varphi} I;
\end{align*}
\item
one-dimensional representations $\theta_{\varphi_1, \varphi_2}$, $\varphi_1, \varphi_2 \in [0, 2 \pi)$:
\begin{align*}
	\theta_{\varphi_1, \varphi_2}(z_{1 1}) &= q^{-1} e^{i \varphi_1} \\
	\theta_{\varphi_1, \varphi_2}(z_{2 1}) &= 0 \\
	\theta_{\varphi_1, \varphi_2}(z_{2 2}) &= e^{i \varphi_2}.
\end{align*}
\end{enumerate}
\end{theorem}
From the above list it readily follows that $\PolMatsymq$ is $*$-bounded, i.e., for each $x \in \PolMatsymq$ there exists a constant $C_x$ such that $\| \pi(x) \| \leq C_x$ for all representations $\pi$ of $\PolMatsymq$. We let $\contq$ denote the universal enveloping $C^*$-algebra of $\PolMatsymq$ and $\holoq$ the closed (non-involutive) subalgebra generated by $z_{1 1}$, $z_{2 1}$, and $z_{2 2}$. We recall that the universal enveloping $C^*$-algebra can be defined as a pair $(\contq, \rho)$, where $\rho: \PolMatsymq \rightarrow \contq$ is a $*$-homomorphism with the property that for each representation $\pi$ of $\PolMatsymq$ there is a unique representation $\varphi$ of $\contq$ such that $\pi = \varphi \circ \rho$. It is useful to note that the irreducible representations of $\PolMatsymq$ are in one-to-one correspondence with the irreducible representations of $\contq$. We say that $\contq$ (resp. $\holoq$) is a $q$-analog of the $C^*$-algebra of continuous functions (resp. subalgebra of holomorphic functions) on the closed unit ball of symmetric complex $2 \times 2$ matrices $\Dsym = \{ \Z \in \Matsym: \Z^* \Z \leq I \}$.

We will now consider an alternative way of constructing representations of $\PolMatsymq$ which was presented in \cite{bershtein-1}. Imperative to this construction is the following $*$-homomorphism, whose existence was indicated in \cite{bershtein-1} without proof, of a coaction corresponding to the action of the unitary group $U_2$ of $2\times 2$ matrices
\[
	Z \mapsto U^T Z U, \quad U \in U_2, \,
	Z \in \Dsym \subset \textnormal{Mat}_2^\sym.
\]
\begin{lemma}
\label{lemma:homomorphism}
There is a $*$-homomorphism
\[
	\mathcal{D}: \PolMatsymq \longrightarrow \PolMatsymq \otimes \CSUq
\]
given by
\[
	\mathcal{D}(z_{i j}) =
	\sum_{k, l = 1}^2 z_{k l} \otimes t_{k i} t_{l j}, \quad i,j=1,2.
\]
\end{lemma}
\begin{proof}
We begin by establishing that the restriction of $\mathcal{D}$ to $\CMatsymq$,
\begin{equation}
\label{eq:homomorphism}
	\CMatsymq \longrightarrow \CMatsymq \otimes \CSLq,
\end{equation}
is a homomorphism. Using the fact that $\CSLq \subset (\Uqsl)^*$ as linear functionals given by~\eqref{eq:CSLq-functionals}, we claim that the map~\eqref{eq:homomorphism} recovers the $\Uqsl$-action on $\CMatsymq$, i.e.,
\begin{equation*}
\label{eq:action}
	\mathcal{D}(x)(\xi) = \xi x, \quad x \in \CMatsymq, \, \xi \in \Uqsl.
\end{equation*}
Consequently $\mathcal{D}$ respects the relations~\eqref{eq:CMatsymq-2},
showing that the map~\eqref{eq:homomorphism} is a well-defined homomorphism.

It is straightforward to verify that the claim holds when $x$ and $\xi$ are generators of $\CMatsymq$ and $\Uqsl$ respectively. In order to show that $\mathcal{D}(z_{i j})(\xi) = \xi z_{i j}$ for all $\xi \in \Uqsl$, it would be enough to see that whenever $\mathcal D(z_{ij})(\xi_k)=\xi_kz_{ij}$ for $\xi_k\in \Uqsl$, $k=1,2$, we have $\mathcal{D}(z_{i j})(\xi_1 \xi_2) = \xi_1 \xi_2 z_{i j}$. Using the fact that the comultiplication is a homomorphism, we have the following computation:
\begin{align*}
	\mathcal{D}(z_{i j})(\xi_1 \xi_2) &=
	\sum_{k, l = 1}^2 z_{k l} t_{k i} t_{l j}(\xi_1 \xi_2)
	= \sum_{k, l = 1}^2 z_{k l} \Delta(t_{k i} t_{l j})(\xi_1 \otimes \xi_2) \\
	&= \sum_{r, s = 1}^2 \sum_{k, l = 1}^2 z_{k l}
	(t_{k r} t_{l s})(\xi_1) (t_{r i} t_{s j})(\xi_2) \\
	&= \sum_{r, s = 1}^2 \mathcal{D}(z_{r s})(\xi_1) (t_{r i} t_{s j})(\xi_2) \\
	&= \xi_1 \sum_{r, s = 1}^2 z_{r s} (t_{r i} t_{s j})(\xi_2)
	= \xi_1 \mathcal{D}(z_{i j})(\xi_2) = \xi_1 \xi_2 z_{i j}.
\end{align*}
It remains to show that extending $\mathcal D$ to $\CMatsymq$ naturally by linearity and by letting $\mathcal{D}(fg)= \mathcal{D}(f) \mathcal{D}(g)$, $f$, $g\in \CMatsymq$, we obtain  $\mathcal{D}(f g)(\xi) = \xi (f g)$ for all $\xi \in \Uqsl$ and  $f$, $g\in\CMatsymq$. Let
\[
	\Delta(\xi) = \sum_k \xi_k^{(1)} \otimes \xi_k^{(2)}
\]
denote the comultiplication of an element $\xi \in \Uqsl$. For $f, g$ generators of $\CMatsymq$ we have
\begin{align*}
	\mathcal{D}(f g)(\xi) &= \mathcal{D}(f) \mathcal{D}(g)(\xi) \\
	&= \left( \sum_i f_i \otimes s_i \right)
	\left( \sum_j g_j \otimes t_j \right)(\xi) \\
	&= \sum_{i, j} f_i g_j s_i t_j(\xi)
	= \sum_k \sum_{i, j} f_i g_j s_i \left(\xi_k^{(1)} \right)
	t_j \left( \xi_k^{(2)} \right) \\
	&= \sum_k \left( \sum_i f_i s_i \left( \xi_k^{(1)} \right) \right)
	\left( \sum_j g_j t_j \left( \xi_k^{(2)} \right) \right) \\
	&= \sum_k \xi_k^{(1)} f \, \xi_k^{(2)} g = \xi (f g).
\end{align*}
The general case is proved by induction on the degree of $f$ and $g$.
Since $\PolMatsymq$ is a $\Uqsu$-module algebra and the involutions in $\Uqsu$ and $\CSUq$ are compatible, it follows that~\eqref{eq:homomorphism} can be extended to a $*$-homo\-morphism on $\PolMatsymq$.
\end{proof}
From relations~\eqref{eq:CMatsymq-2}--\eqref{eq:PolMatsymq-last} it follows that the family of maps
\[
	\Pi_\varphi: \PolMatsymq \longrightarrow \PolCq,
\]
$\varphi \in [0, 2 \pi)$, defined on the generators of $\PolMatsymq$ by
\[
	\Pi_\varphi(z_{1 1}) = q^{-1} z, \quad
	\Pi_\varphi(z_{2 1}) = 0, \quad
	\Pi_\varphi(z_{2 2}) = e^{i \varphi}
\]
is a $*$-homomorphism. 

Let $\rho_F$ and  $\rho_\varphi$, $\varphi\in[0,2\pi)$,  be the irreducible representations of $\PolCq$ given in Section \ref{section2}.  Defining
\[
	\mathcal{F}_\varphi = \rho_F \circ \Pi_\varphi, \quad
	\chi_{\varphi_1, \varphi_2} = \rho_{\varphi_1} \circ \Pi_{\varphi_2},
\]
 we obtain two families of representations of $\PolMatsymq$:
\[
	(\mathcal{F}_\varphi \otimes \pi_0) \circ \mathcal{D}, \quad
	(\chi_{\varphi_1, \varphi_2} \otimes \pi_0) \circ \mathcal{D},\quad
\varphi, \varphi_1, \varphi_2 \in [0, 2 \pi),
\]
 here $\pi_0$ is the irreducible representation of $C[SU_2]_q$ given by (\ref{suq2}). Evaluated on the generators, we have
\begin{align*}
	(\mathcal{F}_\varphi \otimes \pi_0) \circ \mathcal{D}(z_{1 1})
	&= q^{-1} \rho_F(z) \otimes \pi_0(t_{1 1})^2 +
	e^{i \varphi} I \otimes \pi_0(t_{2 1})^2 \\
	&= q^{-1} C_4 S \otimes S^* C_2 S^* C_2 + e^{i \varphi} I \otimes D^2
\end{align*}
\begin{align*}
	(\mathcal{F}_\varphi \otimes \pi_0) \circ \mathcal{D}(z_{2 1})
	&= q^{-1} \rho_F(z) \otimes \pi_0(t_{1 2}) \pi_0(t_{1 1}) + e^{i \varphi} I \otimes 		\pi_0(t_{2 2}) \pi_0(t_{2 1}) \\
	&= -q^{-1} C_4 S \otimes S^* C_2 D + e^{i \varphi} I \otimes C_2 S D
\end{align*}
\begin{align*}
	(\mathcal{F}_\varphi \otimes \pi_0) \circ \mathcal{D}(z_{2 2})
	&= q^{-1} \rho_F(z) \otimes \pi_0(t_{1 2})^2 +
	e^{i \varphi} I \otimes \pi_0(t_{2 2})^2 \\
	&= q C_4 S \otimes D^2 + e^{i \varphi} I \otimes C_2 S C_2 S
\end{align*}
and
\begin{align}
\begin{split}
\label{eq:chi}
	(\chi_{\varphi_1, \varphi_2} \otimes \pi_0) \circ \mathcal{D}(z_{1 1}) &=
	q^{-1} e^{i \varphi_1} S^* C_2 S^* C_2 + e^{i \varphi_2} D^2 \\
	(\chi_{\varphi_1, \varphi_2} \otimes \pi_0) \circ \mathcal{D}(z_{2 1}) &=
	-q^{-1} e^{i \varphi_1} S^* C_2 D + e^{i \varphi_2} C_2 S D \\
	(\chi_{\varphi_1, \varphi_2} \otimes \pi_0) \circ \mathcal{D}(z_{2 2}) &=
	q e^{i \varphi_1} D^2 + e^{i \varphi_2} C_2 S C_2 S.
\end{split}
\end{align}
\begin{lemma}
\label{lemma:wick}
The representation $(\mathcal{F}_\varphi \otimes \pi_0) \circ \mathcal{D}$, $\varphi \in [0, 2 \pi)$, is unitarily equivalent to $\tau_\varphi$.
\end{lemma}
\begin{proof}
It is straightforward to verify that $\Omega = e_0 \otimes e_0$ is cyclic for all representations $\tau_\varphi$ and $(\mathcal{F}_\varphi \otimes \pi_0) \circ \mathcal{D}$, $\varphi \in [0, 2 \pi)$, and
\begin{align*}
	\tau_\varphi(z_{1 1})^* \Omega &=
	(\mathcal{F}_\varphi \otimes \pi_0) \circ \mathcal{D}(z_{1 1})^* \Omega =
	e^{-i \varphi} \Omega \\
	\tau_\varphi(z_{2 1})^* \Omega &=
	(\mathcal{F}_\varphi \otimes \pi_0) \circ \mathcal{D}(z_{2 1})^* \Omega = 0 \\
	\tau_\varphi(z_{2 2})^* \Omega &=
	(\mathcal{F}_\varphi \otimes \pi_0) \circ \mathcal{D}(z_{2 2})^* \Omega = 0.
\end{align*}
Therefore both $\tau_\varphi$ and $(\mathcal{F}_\varphi \otimes \pi_0) \circ \mathcal{D}$ are coherent representations of the Wick algebra corresponding to $\PolMatsymq$ with equal coherent state. (We refer to~\cite{wick} for the definition and properties of coherent representations of $*$-algebras allowing Wick ordering.) Since a coherent representation of a Wick algebra is unique up to unitary equivalence by~\cite{wick}*{Proposition~1.3.3}, this proves the lemma.
\end{proof}
\begin{theorem}
\label{thm:fock-isomorphism}
The Fock representation $\pi_F$ of $\contq$ is faithful, and consequently $\contq$ is $*$-isomorphic to $C^*(\pi_F(\PolMatsymq))$.
\end{theorem}
\begin{proof} Let $C^*(S)$ be the $C^*$-algebra generated by the isometry $S$. 
Recall that for $\varphi \in [0, 2 \pi)$, there exists a $*$-homomorphism $\Theta_\varphi: C^*(S) \rightarrow \mathbb{C}$ defined by $\Theta_\varphi(S) = e^{i \varphi}$, see e.g. \cite{davidson}.

The operators in~\eqref{eq:operators} satisfy
\begin{align*}
\label{eq:operators-identity-1}
C_n^2 &= (1 - q^n) \sum_{k = 0}^\infty q^{n k} S^{k + 1} (S^*)^{k + 1} \\
D &= \sum_{k = 0}^\infty q^k \left( S^k (S^*)^k - S^{k + 1} (S^*)^{k + 1} \right),
\end{align*}
and hence $C_n, D \in C^*(S)$. Moreover, we have $\Theta_\varphi(C_n) = 1$ and $\Theta_\varphi(D) = 0$.

We note that $C^*(\pi_F(\PolMatsymq)) \subset C^*(S)^{\otimes 3}$ and similarly for the other representations. By letting $\Theta_\varphi$ act on the last factor in the tensor products, we get the induced $*$-homomorphisms
\begin{align*}
	C^*(\pi_F(\PolMatsymq))
	\xrightarrow{I \otimes I \otimes \Theta_\varphi}
	&C^*(\tau_\varphi(\PolMatsymq)) \longrightarrow  \\
	& \xrightarrow{I \otimes \Theta_\varphi}
	C^*(\omega_\varphi(\PolMatsymq)).
\end{align*}
Since $C^*(\tau_\varphi(\PolMatsymq))$ is $*$-isomorphic to $C^*((\mathcal{F}_\varphi \otimes \pi_0) \circ \mathcal{D}(\PolMatsymq))$ by Lemma~\ref{lemma:wick}, by letting $\Theta_0$ act on the last factor in the tensor product for $(\mathcal{F}_\varphi \otimes \pi_0) \circ \mathcal{D}$, we get an induced $*$-homomorphism
\[
	C^*(\tau_\varphi(\PolMatsymq)) \longrightarrow
	C^*(\nu_\varphi(\PolMatsymq)).
\]
Finally, by letting $\Theta_{\varphi_1}$ act on $\nu_{\varphi_2}$, $\varphi_1, \varphi_2 \in [0, 2 \pi)$, we get an induced $*$-homomorphism
\[
	C^*(\nu_{\varphi_2}(\PolMatsymq)) \xrightarrow{\Theta_{\varphi_1}}
	C^*(\theta_{\varphi_1, \varphi_2}(\PolMatsymq)).
\]
As a $*$-homomorphism between $C^*$-algebras is contractive we get that for all $x \in \PolMatsymq$ and all irreducible representations $\pi$ of $\PolMatsymq$, $\| \pi(x) \| \leq \| \pi_F(x) \|$. By the definition of $\contq$, it follows that the $*$-homomorphism
\[
	\pi_F: \contq \longrightarrow C^*(\pi_F(\PolMatsymq))
\]
is an isomorphism.
\end{proof}
\section{The Shilov boundary} 
The notion of a noncommutative analog of the maximum modulus principle goes back to the foundational paper \cite{arveson1} by W. Arveson. Recall that the Shilov boundary of a compact Hausdorff space $X$ relative to a uniform algebra $\mathcal{A}$ in $C(X)$ is the smallest closed subset $S \subset X$ such that every function in $\mathcal{A}$ attains its  maximum modulus on $S$. The prototypical example of this is of course the maximum modulus principle encountered in the theory of holomorphic functions. For the disk algebra $A(\mathbb{D}) \subset C(\mathbb{D})$, consisting of functions that are continuous on the closed unit disk $\mathbb{D}$ and holomorphic on its interior, it is well known that every function in $A(\mathbb{D})$ attains its maximum modulus on the unit circle $\mathbb{T}$.

When passing to the noncommutative setting, a notion that arises is that of completely contractive and completely isometric maps. Let $E$ be a subspace of a $C^*$-algebra $\mathcal B$, and let $M_n(E)$ be the space of $n \times n$-matrices with entries in $E$ and norm induced by the one on $M_n(\mathcal B)$. Then any linear map $T$ from $E$ to another $C^*$-algebra $\mathcal{C}$ induces a linear map $T^{(n)}: M_n(E) \to M_n(\mathcal{C})$ by letting
\[
	T^{(n)}((a_{ij}))=(T(a_{ij})), \quad (a_{ij})\in M_n(E).
\]
The linear map $T$ is called a contraction (resp. an isometry) if $\|T\|\leq 1$ (resp. $\|T(a)\|=\|a\|$ for any $a\in E$).
It is called a complete contraction (resp. a complete isometry) if $T^{(n)}$ is a contraction (resp. an isometry) for all $n\in\mathbb N$.
Clearly a  $*$-homomorphism between $C^*$-algebras is completely contractive. 

The following noncommutative generalization of the Shilov boundary was given by Arveson in \cite{arveson1}.
\begin{definition}
Let $\mathcal{A}$ be a subspace of a $C^*$-algebra $\mathcal{B}$ such that $\mathcal{A}$ contains the identity of $\mathcal{B}$ and generates $\mathcal{B}$ as a $C^*$-algebra. A closed ideal $\mathcal{J}$ in $\mathcal{B}$ is called a \emph{boundary ideal} for $\mathcal{A}$ if the canonical quotient map $j_q: \mathcal{B} \rightarrow \mathcal{B} / \mathcal{J}$ is a complete isometry when restricted to $\mathcal{A}$. A boundary ideal is called the \emph{Shilov boundary} for $\mathcal{A}$ if it contains every other boundary ideal.
\end{definition}

It is clear from the definition that if the Shilov boundary exists, then it is unique, and it was shown by M. Hamana in~\cite{hamana} that the Shilov boundary exists for any $\mathcal A$ satisfying the conditions of the above definition. It is not difficult to see that this definition is equivalent to the definition of the Shilov boundary given above in the commutative case, i.e., when $\mathcal{B} = C(X)$.
\begin{example}
The ideal $\mathcal{J} = \{ f \in C(\mathbb{D}): f|_\mathbb{T} = 0 \}$ is the Shilov boundary for $A(\mathbb{D})$.
\end{example}
\begin{example}
In \cite{pro-tur}, the authors considered a $q$-analog $C(\mathbb{D}_2)_q$ (resp. $A(\mathbb{D}_2)_q$) of the $C^*$-algebra of continuous functions (resp. subalgebra of holomorphic functions) on the closed unit ball of complex $2 \times 2$ matrices $\mathbb{D}_2 = \{ \Z \in \textnormal{Mat}_2: \Z^* \Z \leq I \}$. The former was defined as the universal enveloping $C^*$-algebra of $\textnormal{Pol}(\textnormal{Mat}_2)_q$, a $q$-analog of the $*$-algebra of polynomials on $\mathbb{D}_2$. It was proven that the ideal in $C(\mathbb{D}_2)_q$ generated by
\[
	\sum_{j = 1}^2 q^{4 - \alpha - \beta} z_j^\alpha (z_j^\beta)^* -
	\delta^{\alpha \beta}, \quad
	\alpha, \beta = 1, 2,
\]
is the Shilov boundary for $A(\mathbb{D}_2)_q$.
\end{example}
Let $J$ be the $*$-ideal of $\PolMatsymq$ generated by
\[
	\sum_{k = 1}^2 q^{4 - i - j}z_{i k} z_{j k}^* -
	\delta_{i j},
	\quad i, j = 1, 2,
\]
and let $\mathcal{J}$ be the closed ideal generated by the image of $J$ in $\contq$. We shall refer to the quotient $\shilovcontq = \contq / \mathcal{J}$ as a $q$-analog of the $C^*$-algebra of continuous functions on the Shilov boundary of $\Dsym$. The canonical quotient map $j_q: \contq \rightarrow \shilovcontq$ is a $q$-analog of the restriction map that sends a continuous function on $\Dsym$ to its restriction to the Shilov boundary $S(\Dsym) = \{ \Z \in \Matsym: \Z^* \Z = I \}$. The aim of this section is to prove that $\mathcal{J}$ is the Shilov boundary for $\holoq$.

From the above discussion of representations of $\PolMatsymq$, we have the following result on which representations annihilate $J$, whose proof is a straightforward verification.
\begin{lemma}
\label{lemma:annihilators}
The representations $\omega_\varphi$ and $\theta_{\varphi_1, \varphi_2}$, $\varphi, \varphi_1, \varphi_2 \in [0, 2 \pi)$, are the only, up to unitary equivalence, irreducible representations of $\PolMatsymq$ that annihilate $J$. Moreover, any representation $(\chi_{\varphi_1, \varphi_2} \otimes \pi_0) \circ \mathcal{D}$, $\varphi_1, \varphi_2 \in [0, 2 \pi)$, annihilates $J$.
\end{lemma}
\begin{theorem}
\label{thm:boundary-ideal}
The ideal $\mathcal{J}$ is a boundary ideal for $\holoq$.
\end{theorem}
\begin{proof}
By Lemma~\ref{lemma:annihilators}, any representation $(\chi_{\varphi_1, \varphi_2} \otimes \pi_0) \circ \mathcal{D}$, $\varphi_1, \varphi_2 \in [0, 2 \pi)$, annihilates $J$. Thus we have a family of $*$-homomorphisms
\[
	\shilovcontq \longrightarrow
	C^*((\chi_{\varphi_1, \varphi_2} \otimes \pi_0) \circ \mathcal{D}(\PolMatsymq))
\]
given by $b + \mathcal{J} \mapsto (\chi_{\varphi_1, \varphi_2} \otimes \pi_0) \circ \mathcal{D}(b)$, and consequently
\[
	\sup_{\varphi_1, \varphi_2 \in [0, 2 \pi)}
	\| ((\chi_{\varphi_1, \varphi_2} \otimes \pi_0) \circ \mathcal{D}(b_{i j})) \| \leq
	\| (b_{i j} + \mathcal{J}) \|
\]
for all $(b_{i j}) \in M_n(\contq)$. Since the quotient map $j_q: \contq \rightarrow \shilovcontq$ is a $*$-homomorphism, $j_q$ and consequently $j_q|_{\holoq}$ is a complete contraction. It is therefore sufficient to prove that
\[
	\| (a_{i j}) \| =
	\| (\pi_F(a_{i j})) \| \leq
	\sup_{\varphi_1, \varphi_2 \in [0, 2 \pi)}
	\| ((\chi_{\varphi_1, \varphi_2} \otimes \pi_0) \circ \mathcal{D}(a_{i j})) \|
\]
for all $(a_{i j}) \in M_n(\holoq)$.

We note that the operator $C_4 S$ is a contraction on $H = \ell^2(\Zp)$. By Sz.-Nagy's dilation theorem (see e.g.~\cite{paulsen}*{Theorem~1.1}), there exists a unitary operator $U$ on a Hilbert space $K$ containing $H$ as a subspace such that $(C_4 S)^n = P_H U^n|_H$ for all $n \geq 0$. Consider the map $\Psi$ into $\mathcal{B}(H^{\otimes 2} \otimes K)$ defined on the generators of $\PolMatsymq$ by
\begin{align*}
	\Psi(z_{1 1}) &= I \otimes D^2 \otimes U - q^{-1} S^* C_4
	\otimes C_2 S C_2 S \otimes I \\
	\Psi(z_{2 1}) &= D^2 \otimes C_2 S \otimes I \\
	\Psi(z_{2 2}) &= C_4 S \otimes I \otimes I.
\end{align*}
It is readily verified that this map extends uniquely to a representation of $\PolMatsymq$ on $H^{\otimes 2} \otimes K$. By the spectral theorem, $\Psi$ can be written as a direct integral representation of the field of representations $\{ \tau_\varphi: \varphi \in [0, 2 \pi) \}$, i.e.,
\[
	\Psi = \int_{[0, 2 \pi)}^\oplus \tau_\varphi \otimes I_\varphi \, d \mu(\varphi).
\]
For $\xi \in H^{\otimes 2} \otimes K$, we have
\[
	\| \Psi(b) \xi \|^2 =
	\int_0^{2 \pi} \| \tau_\varphi
	\otimes I_\varphi(b) \xi(\varphi) \|^2 \, d \mu(\varphi) \leq
	\sup_{\varphi \in [0, 2 \pi)} \| \tau_\varphi(b) \|^2 \| \xi \|^2.
\]
Thus $\| \Psi(b) \| \leq \sup_{\varphi \in [0, 2 \pi)} \| \tau_\varphi(b) \|$
for all $b \in \contq$, and since $\Psi$ induces a representation on $M_n(\contq)$, similar arguments show that
\[
	\| (\Psi(b_{i j})) \| \leq
	\sup_{\varphi \in [0, 2 \pi)} \| (\tau_\varphi(b_{i j})) \|
\]
for all $(b_{i j}) \in M_n(\contq)$. Since $\pi_F(a) = (I \otimes I \otimes P_H) \Psi(a)|_{H^{\otimes 3}}$, we get
\begin{align}
\label{eq:inequality-1}
	\| (\pi_F(a_{i j})) \| \leq
	\sup_{\varphi \in [0, 2 \pi)} \| (\tau_\varphi(a_{i j})) \|
\end{align}
for all $(a_{i j}) \in M_n(\holoq)$.

Our next step is to show that, for all $\varphi \in [0, 2 \pi)$,
\[
	\| (\tau_\varphi(a_{i j})) \| \leq
	\sup_{\varphi_1, \varphi_2 \in [0, 2 \pi)}
	\| ((\chi_{\varphi_1, \varphi_2} \otimes \pi_0) \circ \mathcal{D}(a_{i j})) \|
\]
for all $(a_{i j}) \in M_n(\holoq)$. Similar to the previous step, we consider the map $\Psi_\varphi$ into $\mathcal{B}(K \otimes H)$ defined on the generators of $\PolMatsymq$ by
\begin{align*}
	\Psi_\varphi(z_{1 1}) &=
	q^{-1} U \otimes S^* C_2 S^* C_2 + e^{i \varphi} I \otimes D^2 \\
	\Psi_\varphi(z_{2 1}) &=
	-q^{-1} U \otimes S^* C_2 D + e^{i \varphi} I \otimes C_2 S D \\
	\Psi_\varphi(z_{2 2}) &= q U \otimes D^2 + e^{i \varphi} I \otimes C_2 S C_2 S.
\end{align*}
It is readily verified that $\Psi_\varphi$ extends to a representation of $\contq$ on $K \otimes H$. By~\eqref{eq:chi} and the spectral theorem, $\Psi_\varphi$ can be written as a direct integral representation of the field of representations $\{ (\chi_{\varphi_1, \varphi} \otimes \pi_0) \circ \mathcal{D}: \varphi_1 \in [0, 2 \pi) \}$, i.e.,
\[
	\Psi_\varphi =
	\int_{\varphi_1 \in [0, 2 \pi)}^\oplus (\chi_{\varphi_1, \varphi} \otimes \pi_0)
	\circ \mathcal{D} \otimes I_{\varphi_1} \, d \mu(\varphi_1).
\]
For $\xi \in K \otimes H$, we have
\begin{align*}
	\| \Psi_\varphi(b) \xi \|^2 &=
	\int_0^{2 \pi} \| (\chi_{\varphi_1, \varphi} \otimes \pi_0) \circ \mathcal{D} \otimes
	I_{\varphi_1}(b) \xi(\varphi_1) \|^2 \, d \mu(\varphi_1) \\
	&\leq \sup_{\varphi_1 \in [0, 2 \pi)} \| (\chi_{\varphi_1, \varphi} \otimes \pi_0)
	\circ \mathcal{D}(b) \|^2 \int_0^{2 \pi} \| \xi(\varphi_1) \|^2 \, d \mu(\varphi_1) \\
	&= \sup_{\varphi_1 \in [0, 2 \pi)} \| (\chi_{\varphi_1, \varphi} \otimes \pi_0)
	\circ \mathcal{D}(b) \|^2 \| \xi \|^2.
\end{align*}
Thus
\[
	\| \Psi_\varphi(b) \| \leq \sup_{\varphi_1 \in [0, 2 \pi)} \|
	(\chi_{\varphi_1, \varphi} \otimes \pi_0) \circ \mathcal{D}(b) \|
\]
for all $b \in \contq$. Since $\Psi_\varphi$ induces a representation on $M_n(\contq)$, similar arguments show that
\[
	\| (\Psi_\varphi(b_{i j})) \| \leq \sup_{\varphi_1 \in [0, 2 \pi)} \|
	((\chi_{\varphi_1, \varphi} \otimes \pi_0) \circ \mathcal{D}(b_{i j})) \|
\]
for all $(b_{i j}) \in M_n(\contq)$. Since
\[
	(\mathcal{F}_\varphi \otimes \pi_0) \circ \mathcal{D}(a) =
	(P_H \otimes I) \Psi_\varphi(a)|_{H^{\otimes 2}}
\]
and
\[
\| \tau_\varphi(a) \| =
\| (\mathcal{F}_\varphi \otimes \pi_0) \circ \mathcal{D}(a) \|
\]
for all $a \in \holoq$, we have
\[
	\| \tau_\varphi(a) \| \leq \| \Psi_\varphi(a) \| \leq
	\sup_{\varphi_1, \varphi_2 \in [0, 2 \pi)}
	\| (\chi_{\varphi_1, \varphi_2} \otimes \pi_0) \circ \mathcal{D}(a) \|.
\]
By a similar argument, we have
\begin{equation}
\label{eq:inequality-2}
	\| (\tau_\varphi(a_{i j})) \| \leq
	\sup_{\varphi_1, \varphi_2 \in [0, 2 \pi)}
	\| ((\chi_{\varphi_1, \varphi_2} \otimes \pi_0) \circ \mathcal{D}(a_{i j})) \|
\end{equation}
for all $(a_{i j}) \in M_n(\holoq)$. By combining the inequalities~\eqref{eq:inequality-1} and~\eqref{eq:inequality-2}, we get the desired statement.
\end{proof}
\begin{lemma}
\label{lemma:bound}
If $\pi$ is a representation of $\PolMatsymq$ that annihilates $J$, then
\[
	\| \pi(x) \| \leq \sup_{\varphi \in [0, 2 \pi)} \| \omega_\varphi(x) \|
\]
for all $x \in \PolMatsymq$.
\end{lemma}
\begin{proof}
%
As $\theta_{\varphi_1,\varphi_2}$ and $\omega_\varphi$, $\varphi, \varphi_1, \varphi_2 \in [0, 2 \pi)$ are the only irreducible representations of $\PolMatsymq$ that annihilate $J$
it is sufficient to prove that
\[
	|\theta_{\varphi_1, \varphi_2}(x)| \leq
	\sup_{\varphi \in [0, 2 \pi)} \| \omega_\varphi(x) \|
\]
for all $x \in \PolMatsymq$ and $\varphi_1, \varphi_2 \in [0, 2 \pi)$.
Recall that $C^*(\omega_\varphi(\PolMatsymq))$ is a subalgebra of $C^*(S)$ and if $\Theta_{\varphi_2}:C^*(S)\to\mathbb C$ is the $*$-homomorphism given by $\Theta_{\varphi_2}(S)=e^{i\varphi_2}$, then
it is readily verified that $\Theta_{\varphi_2}$ induces a $*$-homomorphism
\[
	C^*(\omega_{(\varphi_1 + \varphi_2 + \pi)/2}(\PolMatsymq)) \longrightarrow
	C^*(\theta_{\varphi_1, \varphi_2}(\PolMatsymq)),
\]
where each generator $\omega_{(\varphi_1 + \varphi_2 + \pi)/2}(z_{i j})$ is mapped to $\theta_{\varphi_1, \varphi_2}(z_{i j})$. Thus
\[
	|\theta_{\varphi_1, \varphi_2}(x)| \leq
	\| \omega_{(\varphi_1 + \varphi_2 + \pi)/2}(x) \| \leq
	\sup_{\varphi \in [0, 2 \pi)} \| \omega_\varphi(x) \|
\]
for all $x \in \PolMatsymq$ and $\varphi_1, \varphi_2 \in [0, 2 \pi)$, which proves the lemma.
\end{proof}
\begin{theorem}
The ideal $\mathcal{J}$ contains all other boundary ideals.
\end{theorem}
\begin{proof}
Let $\mathcal{I}$ be a boundary ideal such that $\mathcal{I} \supset \mathcal{J}$, and let $i_q$ and $j_q$,
\begin{align*}
	i_q&: \contq \longrightarrow \contq / \mathcal{I} \\
	j_q&: \contq \longrightarrow \shilovcontq,
\end{align*}
be the canonical quotient maps. If
\[
	K = \left\{ \varphi \in [0, 2 \pi): \omega_\varphi(\mathcal{I}) = 0 \right\}
\]
is nonempty, then $\mathcal{I} \subset \cap_{\varphi \in K} \ker \omega_\varphi$, and hence $\mathcal{I} = \mathcal{J}$ if
\[
	\bigcap_{\varphi \in K} \ker \omega_\varphi \subset \mathcal{J}.
\]
We claim that it is sufficient to prove that $K$ is dense in $[0, 2 \pi]$. Indeed, suppose that $x$ lies in $\ker \omega_\varphi$ for all $\varphi \in K$. If $\bar{K}= [0, 2 \pi]$, it follows by Lemma~\ref{lemma:bound} that $j_q(x) = 0$, i.e., $x \in \mathcal{J}$.

Since $i_q$ and $j_q$ are isometries when restricted to $\holoq$, we have
\begin{equation}
\label{eq:isometry}
	\| z_{i j} + e^{i \theta} + \mathcal{J} \| =
	\| z_{i j} + e^{i \theta} \| =
	\| z_{i j} + e^{i \theta} + \mathcal{I} \|
\end{equation}
for any $\theta \in [0, 2 \pi)$. Since $\omega_\varphi$ and $\theta_{\varphi_1, \varphi_2}, \varphi, \varphi_1, \varphi_2 \in [0, 2 \pi)$, are the only, up to unitary equivalence, irreducible representations of $\PolMatsymq$ that annihilate $J$, Lemma~\ref{lemma:bound} gives
\begin{equation}
\label{eq:maximum}
\begin{aligned}
	\| z_{2 1} + e^{i \theta} + \mathcal{J} \|
	&= \sup_{\varphi \in [0, 2 \pi)} \| \omega_\varphi(z_{2 1}) + e^{i \theta} \| \\
	&= \sup \left\{ |\zeta + e^{i \theta}|:
	\zeta \in \bigcup_{k \geq 0} q^{2 k} \mathbb{T} \right\}=2.
\end{aligned}
\end{equation}
If $\pi$ is an irreducible representation of $\contq / \mathcal{I}$ which does not vanish on $z_{2 1} + \mathcal{I}$, then $\pi \circ i_q$ is an irreducible representation of $\contq$ which does not vanish on $z_{2 1}$. Since $\pi \circ i_q(\mathcal{J}) = 0$, $\pi \circ i_q$ is unitary equivalent to $\omega_\varphi$ for some $\varphi \in K$. Thus
\begin{align*}
	\| z_{2 1} + e^{i \theta} + \mathcal{I} \|
	&= \sup_\pi \| \pi \circ i_q(z_{2 1}) + e^{i \theta} \| \\
	&= \sup_{\varphi \in K}
	\| \omega_\varphi(z_{2 1}) + e^{i \theta} \| \\
	&= \sup \left\{ |\zeta + e^{i \theta}|:
	\zeta \in \bigcup_{k \geq 0} q^{2 k} X_K \right\},
\end{align*}
where $\pi$ ranges over the irreducible representations of $\contq / \mathcal{I}$ and  $X_K = \{ e^{i \varphi}: \varphi \in K \}\subset\mathbb T$. From~\eqref{eq:isometry} and~\eqref{eq:maximum} we conclude that
\[
	\sup \left\{ |\zeta + e^{i \theta}|:
	\zeta \in \bigcup_{k \geq 0} q^{2 k} X_K \right\} = 2
\]
for any $\theta \in [0, 2 \pi)$, and hence $X_K$ must be dense in $\mathbb T$, which proves the theorem.
\end{proof}
\section{Regular functions on the Shilov boundary}
In~\cite{bershtein-2}, a $*$-algebra $\C[S(\mathbb{D}_n^\sym)]_q$ referred to as the algebra of regular functions on the Shilov boundary on the quantum unit ball in the space of symmetric complex $n \times n$ matrices was defined as the localization of $\C[\textnormal{Mat}_n^\sym]_q$ with respect to the Ore system $(\detqsym \mathbf{z})^\Zp$, where $\detqsym \mathbf{z}$ is a $q$-analog of the determinant of the symmetric matrix $\mathbf{z} = (z_{i j})$ corresponding to the generators of $\C[\textnormal{Mat}_n^\sym]_q$ (see~\cite{bershtein-2} for definitions of $\C[\textnormal{Mat}_n^\sym]_q$ and $\detqsym \mathbf{z}$). In this section we show that for our particular case, $n = 2$, this agrees with our previous result.

In our case of $\shilovregq$ the quantum determinant takes the form
\[
	\detqsym \mathbf{z} = z_{2 2} z_{1 1} - q^{-1} z_{2 1}^2,
\]
and the involution is given by
\begin{align*}
	z_{1 1}^* &= q^{-2} z_{2 2} (\detqsym \mathbf{z})^{-1} \\
	z_{2 1}^* &= -q^{-1} z_{2 1} (\detqsym \mathbf{z})^{-1} \\
	z_{2 2}^* &= z_{1 1} (\detqsym \mathbf{z})^{-1}.
\end{align*}
\begin{theorem}
The map $k: z_{i j} + \mathcal{J} \mapsto z_{i j} \in \shilovregq$, $i, j = 1, 2$, can be extended to a $*$-isomorphism of the $*$-subalgebra of $\shilovcontq$ generated by $z_{i j} + \mathcal{J}$, $i, j = 1, 2$, onto $\shilovregq$.
\end{theorem}
\begin{proof}
It is straightforward to verify that an extension of $k$ to polynomials in $z_{i j} + \mathcal{J}$, $i, j = 1, 2$, is well-defined. We construct an inverse to $k$ as follows. Since $\pi((\detqsym \mathbf{z})^* \detqsym \mathbf{z}) = q^{-2}$ for all representations of $\PolMatsymq$ that annihilate $J$, it follows that $(\detqsym \mathbf{z})^* \detqsym \mathbf{z} =\detqsym \mathbf{z}(\detqsym \mathbf{z})^* =q^{-2}$ in $\shilovcontq$. Moreover, each $z_{i j}^*$ in $\shilovcontq$ has the same expression in terms of the generators and $(\detqsym \mathbf{z})^{-1}$ as in $\shilovregq$. Since $\shilovregq$ is generated by $z_{i j}$, $i, j = 1, 2$, and $(\detqsym \mathbf{z})^{-1}$, we have a $*$-homomorphism $k': \shilovregq \rightarrow \shilovcontq$ given by $z_{i j} \mapsto z_{i j} + \mathcal{J}$ and $(\detqsym \mathbf{z})^{-1} \mapsto q^2 (\detqsym \mathbf{z})^* + \mathcal{J}$. It is easily verified that $k$ and $k'$ are mutually inverse to each other.
\end{proof}
\section*{Acknowledgments} The authors are grateful to Olga Bershtein for discussions during the preparation of this paper.

\end{document}